\documentclass{article}

\usepackage[unicode=true,
  linktocpage,
  linkbordercolor={0.5 0.5 1},
  citebordercolor={0.5 1 0.5},
  linkcolor=blue]{hyperref}
\usepackage{amsmath, amsfonts, amssymb, amsthm}
\usepackage{array, makecell}
\usepackage{fge}
\usepackage{enumitem}
\usepackage{graphicx}
\usepackage{afterpage}
\usepackage{float}
\usepackage{longtable}
\usepackage{pdflscape}
\usepackage[numbers,sort&compress]{natbib}
\usepackage{psfrag}
\usepackage{parskip}
\usepackage[toc,page]{appendix}
\usepackage{tensor}
\usepackage{tikz,tikz-cd}
\usetikzlibrary{decorations.pathmorphing,hobby}
\usepackage{thmtools,thm-restate}
\usepackage{circledtext}
\usepackage[dvipsnames]{xcolor}

\usepackage{nameref}

\usepackage[margin=0.85in]{geometry}

\begingroup
    \makeatletter
    \@for\theoremstyle:=definition,remark,plain\do{%
        \expandafter\g@addto@macro\csname th@\theoremstyle\endcsname{%
            \addtolength\thm@preskip\parskip
            }%
        }
\endgroup

\theoremstyle{plain}

\theoremstyle{definition}

\newtheorem{propn}{Proposition}[section]
\newtheorem{rmk}{Remark}[section]

\theoremstyle{remark}

\newcommand\numberthis{\addtocounter{equation}{1}\tag{\theequation}}

\numberwithin{equation}{section}

\hypersetup{
    colorlinks=true,
    linkcolor=BrickRed,
    filecolor=Apricot,      
    urlcolor=Rhodamine,
    citecolor=ForestGreen,
    anchorcolor=Fuchsia,
}

\newcommand{\bb}[1]{
    \mathbb{#1}
}

\newcommand{\wtilde}[1]{
    \widetilde{#1}
}

\newcommand{\eps}{
    \epsilon
}

\newcommand{\rarr}{
    \rightarrow
}
\newcommand{\del}{
    \partial
}
\newcommand{\delbar}{
    \overline{\partial}
}
\newcommand{\nablabar}{
    \overline{\nabla}
}

\newcommand{\inner}[2]{
    \langle #1, #2 \rangle
}

\newcommand{\Binner}[2]{
    \Big\langle #1, #2 \Big\rangle
}

\newcommand{\tr}{
    \text{\normalfont{tr}}
}
\newcommand{\Rm}{
    \text{\normalfont{Rm}}
}

\newcommand{\vol}{
    \text{\normalfont{vol}}
}

\renewcommand{\Re}{
    \text{\normalfont{Re}}\,
}

\newcommand{\cf}{
    \text{\emph{cf.,}}
}

\title{Anomaly Flow II: Estimates and Long-Time Existence for the Coupled Flow}
\author{Caleb Suan \\ \texttt{kwsuan@math.cuhk.edu.hk}}
\date{}

\begin{document}

\maketitle

\begin{abstract}
\label{abstract}
    In this follow-up to \cite{Sua25}, we extend the results to the coupled anomaly flow which incorporates the evolution of a Hermitian metric on a holomorphic vector bundle over a compact complex $3$-fold with general slope parameter $\alpha'$. This yields a similar sufficient smallness condition on the slope parameter $\alpha'$ in order to extend the flow on $[0,\tau)$ to $[0,\tau + \eps)$. We also augment the argument of \cite{Sua25} slightly to fix the peculiarity of the units in the condition for $\alpha'$ for the $k = 1$ case.
\end{abstract}

\tableofcontents

\section{Introduction}
\label{sect-intro}

This paper is a follow-up to the author's previous one \cite{Sua25} on the anomaly flow. Here, we consider a coupled version of the flow, which includes the effects of a gauge bundle $E$ over our base Calabi--Yau $3$-fold $X$.

Consider a compact complex $3$-fold $X$ equipped with a nowhere-vanishing holomorphic $(3,0)$-form $\Omega$ and a holomorphic vector bundle $E \rarr X$. Recall that for a real parameter $\alpha'$, the (coupled) anomaly flow is a geometric flow of a pair of Hermitian metrics $\omega$ on $X$ and $H$ on $E$ whose stationary points are solutions to the Hull--Strominger system \cite{Hul86,Str86}:
\begin{gather*}
    F^{2,0} = F^{0,2} = 0, \qquad \omega^2 \wedge F^{1,1} = 0, \numberthis \label{eqn-HS-1} \\
    \sqrt{-1} \del \delbar \omega = \alpha' \big( \tr (\Rm \wedge \Rm) - \tr (F \wedge F) \big), \numberthis \label{eqn-HS-2} \\
    d \big( |\Omega|_\omega \omega^2 \big) = 0. \numberthis \label{eqn-HS-3}
\end{gather*}
Here $\Rm$ and $F$ are the Chern curvatures of $\omega$ and $H$ respectively, though some authors \cite{Puj21, PU21} use other connections instead.

The evolution equations for the flow are given by
\begin{align}
\label{eqn-anomaly-flow}
    \begin{cases}
        H^{-1} \big( \del_t H \big) &= - \Big( \frac{1}{2 |\Omega|_\omega} \Big) \Lambda_\omega F, \\
        \del_t \big( |\Omega|_\omega \omega^2 \big) &= \sqrt{-1} \del \delbar \omega - \alpha' \big( \tr (\Rm \wedge \Rm) - \tr (F \wedge F) \big).
    \end{cases}
\end{align}
By Chern--Weil theory, we see that the conformally balanced condition \eqref{eqn-HS-3} is preserved given a conformally balanced initial condition for $\omega$. For more on the anomaly flow and its generalisations in various settings, we refer the reader to \cite{GFMPS25, KPS26, PPZ18a, PPZ18b, PPZ18c, PPZ19a, FP21, FHP21, PU21, Puj21, BV22, AMP24}.

\begin{rmk}
    We note that the original formulation of the anomaly flow in \cite{PPZ18c} does not include the dilaton factor $\Big( \frac{1}{2 |\Omega|_\omega} \Big)$ in the evolution equation for $H$. We add this factor in for convenience and to have the evolution equation for $F$ in line with those for $\Rm$ and $T$ (see \S \ref{sect-evol-eqns} below). Since $\Omega$ is nowhere-vanishing, the modified flow still has the same fixed points and also retains short-time existence and uniqueness.
\end{rmk}

Our main result is the following:

\begin{restatable}{thm}{thmmain}
\label{thm-main}
    Suppose that there exist positive constants $B_{min}$, $B_{max}$, $C_{-1}$, and $C_0$ such that
    \begin{gather*}
        B_{min} \leq \Big( \frac{1}{2 |\Omega|_\omega} \Big) \leq B_{max}, \\
        |T|, |\overline{T}| \leq C_{-1}, \qquad |\Rm|, |DT|, |D\overline{T}|, |F| \leq C_0
    \end{gather*}
    along the anomaly flow \eqref{eqn-anomaly-flow} on $t \in [0,\tau)$. There exists a positive constant $M'$ such that if
    \begin{align*}
        \alpha' < M' \cdot \frac{B_{min}^2}{a_0 B_{max}^2 C_0},
    \end{align*}
    then the flow can be extended to $[0, \tau + \eps)$ for some $\eps > 0$.
\end{restatable}

\begin{rmk}
    For simplicity, we shall implicitly assume that $\alpha'$ is non-negative. The results still hold when $\alpha' < 0$ if we replace all instances of $\alpha'$ with $|\alpha'|$.
\end{rmk}

Much of the methodology is the same as that of \cite{Sua25}. As such, we omit most of the steps and only present the results of the major computations.

\textbf{Acknowledgements:} The author thanks S\'{e}bastien Picard and Man Chun Lee for helpful discussions and comments.

\section{Evolution Equations}
\label{sect-evol-eqns}

We first briefly review the relevant evolution equations for the anomaly flow \eqref{eqn-anomaly-flow}, originally derived by \cite{PPZ18b}. Using the schematic equations from \S 2 of \cite{Sua25} and substituting in our choice of $\Phi = \tr (F  \wedge F)$, we have
\begin{align*}
    \del_t g &= \Big( \frac{1}{2 |\Omega|_\omega} \Big) \cdot \Big[ \Rm + T * \overline{T} + \alpha' \big( \Rm * \Rm + F * F \big) \Big], \numberthis \label{eqn-del_t-g} \\[5pt]
    \del_t \Rm &= \Big( \frac{1}{2 |\Omega|_\omega} \Big) \cdot \Big[ \frac{1}{2} \Delta_R \Rm + H_1 + \alpha' \big( \nabla \nablabar (\Rm * \Rm) + \nabla \nablabar (F * F) + H_2 \big) \Big], \numberthis \label{eqn-del_t-Rm} \\[5pt]
    \del_t T &= \Big( \frac{1}{2 |\Omega|_\omega} \Big) \cdot \Big[ \frac{1}{2} \Delta_R T + K_1 + \alpha' \big( \nabla (\Rm * \Rm) + \nabla (F * F) + K_2 \big) \Big]. \numberthis \label{eqn-del_t-T}
\end{align*}

Here $\Delta_R$ is the Laplacian operator given in coordinates by $g^{i\overline{j}} \big( \nabla_i \nablabar_{\overline{j}} + \nablabar_{\overline{j}} \nabla_i \big)$. The $*$ notation above denotes finitely many linear contractions using the metrics $\omega$ and $H$ while $\nabla$ denotes covariant derivatives induced by their Chern connections. The terms $H_1, H_2$, and $K_1, K_2$ are given schematically by
\begin{align*}
    H_1 &= \nabla \nablabar (T * \overline{T}) + \nablabar (T * \Rm) + \nabla (\overline{T} * \Rm) + \Rm * \Rm \\
    &\quad + \nablabar (T * T * \overline{T}) + \nabla (\overline{T} * \overline{T} * T) + T * \overline{T} * \Rm + T * T * \overline{T} * \overline{T}, \\[5pt]
    H_2 &= \nablabar (T * \Rm * \Rm) + \nabla (\overline{T} * \Rm * \Rm) + \nablabar (T * F * F) + \nabla (\overline{T} * F * F) \\
    &\quad + \Rm * \Rm * \Rm + \Rm * F * F  + T * \overline{T} * \Rm * \Rm + T * \overline{T} * F * F,
\end{align*}
and
\begin{align*}
    K_1 &= \nabla (T * \overline{T}) + T * \Rm + T * T * \overline{T}, \\[5pt]
    K_2 &= T * \Rm * \Rm + T * F * F.
\end{align*}

The important thing to keep track of for these terms are the number of covariant derivatives on $\Rm, T$ and $\overline{T}$, and $F$. We see from the above that
\begin{itemize}
    \item $H_1, H_2$: at most $1$ derivative on $\Rm$, $2$ derivatives on $T$ and $\overline{T}$, and $1$ derivative on $F$,

    \item $K_1, K_2$: at most $0$ derivatives on $\Rm$, $1$ derivative on $T$ and $\overline{T}$, and $0$ derivatives on $F$.
\end{itemize}
We group these terms together because they are generically well-behaved for our purposes and leaves behind the ``bad" curvature terms with too many non-Laplacian derivatives.

For the evolution of the curvature $F$, we have the following proposition:

\begin{propn}
    Under the anomaly flow \eqref{eqn-anomaly-flow}, the curvature $F$ evolves by
    \begin{align*}
        \del_t \tensor{F}{_i_{\overline{j}}^\alpha_\beta} = \Big( \frac{1}{2 |\Omega|_\omega} \Big) \cdot &\Big[ g^{p\overline{q}} \nabla_p \nablabar_{\overline{q}} \tensor{F}{_i_{\overline{j}}^\alpha_\beta} + g^{p \overline{q}} \tensor{R}{_p_{\overline{j}}^l_i} \tensor{F}{_l_{\overline{q}}^\alpha_\beta} - g^{p \overline{q}} \tensor{R}{_p_{\overline{j}}_{\overline{q}}^{\overline{l}}} \tensor{F}{_i_{\overline{l}}^\alpha_\beta} \\
        &\quad - g^{p \overline{q}} \tensor{F}{_p_{\overline{j}}^\alpha_\gamma} \tensor{F}{_i_{\overline{q}}^\gamma_\beta} + g^{p \overline{q}} \tensor{F}{_p_{\overline{j}}^\gamma_\beta} \tensor{F}{_i_{\overline{q}}^\alpha_\gamma} + g^{p \overline{q}} \nablabar_{\overline{j}} \big[ \tensor{T}{_i_p^l} \tensor{F}{_l_{\overline{q}}^\alpha_\beta} \big] - g^{p \overline{q}} \nabla_p \big[ \tensor{\overline{T}}{_{\overline{q}}_{\overline{j}}^{\overline{m}}} \tensor{F}{_i_{\overline{m}}^\alpha_\beta} \big] \\
        &\qquad + g^{p\overline{q}} \overline{T}_{\overline{j}} \nabla_i \tensor{F}{_p_{\overline{q}}^\alpha_\beta} + g^{p\overline{q}} T_i \nablabar_{\overline{j}} \tensor{F}{_p_{\overline{q}}^\alpha_\beta} + g^{p\overline{q}} \big[ \nablabar_{\overline{j}} T_i \big] \tensor{F}{_p_{\overline{q}}^\alpha_\beta} \Big] . \numberthis \label{eqn-del_t-F-coords}
    \end{align*}
\end{propn}

\begin{proof}
    Since the complex structure on $X$ and holomorphic structure on $E$ are fixed, we can use that under the general flow
    \begin{align*}
        H^{-1} \big( \del_t H \big) = A
    \end{align*}
    the curvature evolves by
    \begin{align*}
        \del_t F = - \nablabar \nabla A.
    \end{align*}

    In our case, we have
    \begin{align*}
        A = \Big( \frac{1}{2 |\Omega|_\omega} \Big) \Lambda_\omega F.
    \end{align*}
    We use coordinates to rewrite this evolution equation.
    
    First, we recall the differential Bianchi identity
    \begin{align*}
        \nabla_p \tensor{F}{_i_{\overline{j}}^\alpha_\beta} - \nabla_i \tensor{F}{_p_{\overline{j}}^\alpha_\beta} = \tensor{T}{_p_i^l} \tensor{F}{_l_{\overline{j}}^\alpha_\beta}, \quad \nablabar_{\overline{q}} \tensor{F}{_i_{\overline{j}}^\alpha_\beta} - \nablabar_{\overline{j}} \tensor{F}{_i_{\overline{q}}^\alpha_\beta} = \tensor{\overline{T}}{_{\overline{q}}_{\overline{j}}^{\overline{m}}} \tensor{F}{_i_{\overline{m}}^\alpha_\beta}
    \end{align*}
    and the fact that
    \begin{align*}
        \nabla_i \Big( \frac{1}{2 |\Omega|_\omega} \Big) = \Big( \frac{1}{2 |\Omega|_\omega} \Big) T_i, \qquad \nablabar_{\overline{j}} \Big( \frac{1}{2 |\Omega|_\omega} \Big) = \Big( \frac{1}{2 |\Omega|_\omega} \Big) \overline{T}_{\overline{j}}.
    \end{align*}
    
    Using these and the Ricci identity, we have
    \begin{align*}
        \tensor{\del_t F}{_i_{\overline{j}}^\alpha_\beta} &= g^{p \overline{q}} \nablabar_{\overline{j}} \nabla_i \Big[ \Big( \frac{1}{2 |\Omega|_\omega} \Big) \cdot \tensor{F}{_p_{\overline{q}}^\alpha_\beta} \Big] \\
        &= \Big( \frac{1}{2 |\Omega|_\omega} \Big) \cdot \Big[ g^{p\overline{q}} \nablabar_{\overline{j}} \nabla_i \tensor{F}{_p_{\overline{q}}^\alpha_\beta} + g^{p\overline{q}} \overline{T}_{\overline{j}} \nabla_i \tensor{F}{_p_{\overline{q}}^\alpha_\beta} + g^{p\overline{q}} T_i \nablabar_{\overline{j}} \tensor{F}{_p_{\overline{q}}^\alpha_\beta} + g^{p\overline{q}} \big[ \nablabar_{\overline{j}} T_i \big] \tensor{F}{_p_{\overline{q}}^\alpha_\beta} \Big]
    \end{align*}

    The highest order term in $F$ is then
    \begin{align*}
        g^{p\overline{q}} \nablabar_{\overline{j}} \nabla_i \tensor{F}{_p_{\overline{q}}^\alpha_\beta} &= g^{p \overline{q}} \nablabar_{\overline{j}} \nabla_p \tensor{F}{_i_{\overline{q}}^\alpha_\beta} + g^{p \overline{q}} \nablabar_{\overline{j}} \big[ \tensor{T}{_i_p^l} \tensor{F}{_l_{\overline{q}}^\alpha_\beta} \big] \\
        &= g^{p \overline{q}} \nabla_p \nablabar_{\overline{j}} \tensor{F}{_i_{\overline{q}}^\alpha_\beta} + g^{p \overline{q}} \tensor{R}{_p_{\overline{j}}^l_i} \tensor{F}{_l_{\overline{q}}^\alpha_\beta} - g^{p \overline{q}} \tensor{R}{_p_{\overline{j}}_{\overline{q}}^{\overline{l}}} \tensor{F}{_i_{\overline{l}}^\alpha_\beta} \\
        &\quad - g^{p \overline{q}} \tensor{F}{_p_{\overline{j}}^\alpha_\gamma} \tensor{F}{_i_{\overline{q}}^\gamma_\beta} + g^{p \overline{q}} \tensor{F}{_p_{\overline{j}}^\gamma_\beta} \tensor{F}{_i_{\overline{q}}^\alpha_\gamma} + g^{p \overline{q}} \nablabar_{\overline{j}} \big[ \tensor{T}{_i_p^l} \tensor{F}{_l_{\overline{q}}^\alpha_\beta} \big] \\
        &= g^{p\overline{q}} \nabla_p \nablabar_{\overline{q}} \tensor{F}{_i_{\overline{j}}^\alpha_\beta} + g^{p \overline{q}} \tensor{R}{_p_{\overline{j}}^l_i} \tensor{F}{_l_{\overline{q}}^\alpha_\beta} - g^{p \overline{q}} \tensor{R}{_p_{\overline{j}}_{\overline{q}}^{\overline{l}}} \tensor{F}{_i_{\overline{l}}^\alpha_\beta} \\
        &\quad - g^{p \overline{q}} \tensor{F}{_p_{\overline{j}}^\alpha_\gamma} \tensor{F}{_i_{\overline{q}}^\gamma_\beta} + g^{p \overline{q}} \tensor{F}{_p_{\overline{j}}^\gamma_\beta} \tensor{F}{_i_{\overline{q}}^\alpha_\gamma} + g^{p \overline{q}} \nablabar_{\overline{j}} \big[ \tensor{T}{_i_p^l} \tensor{F}{_l_{\overline{q}}^\alpha_\beta} \big] - g^{p \overline{q}} \nabla_p \big[ \tensor{\overline{T}}{_{\overline{q}}_{\overline{j}}^{\overline{m}}} \tensor{F}{_i_{\overline{m}}^\alpha_\beta} \big],
    \end{align*}
    which yields the desired expression.

\end{proof}

From this, our schematic evolution equation for $F$ is
\begin{align}
    \del_t F = \Big( \frac{1}{2 |\Omega|_\omega} \Big) \cdot \Big[ \frac{1}{2} \Delta_R F + \Rm * F + F * F + \nablabar (T * F) + \nabla (\overline{T} * F) \Big], \label{eqn-del_t-F}
\end{align}
where we have used the Ricci identity on the expression \eqref{eqn-del_t-F-coords} to convert $\Delta$ into $\frac{1}{2} \Delta_R$.

From this, we get the following expressions for the evolution of covariant derivatives of $\Rm$, $T$, and $F$:
\begin{align*}
    \del_t \big( \nabla^m \nablabar^l \Rm \big) &= \sum_{i+j>0} \sum_{i=0}^m \sum_{j=0}^l (\nabla^{m-i} \nablabar^{l-j} \Rm) * \big( \nabla^i \nablabar^j (\del_t g) \big) \\
    &\quad + \nabla^m \nablabar^l \bigg( \Big( \frac{1}{2 |\Omega|_\omega} \Big) \cdot \Big[ \frac{1}{2} \Delta_R \Rm + H_1 + \alpha' \big( \nabla \nablabar (\Rm * \Rm) + \nabla \nablabar (F * F) + H_2 \big) \Big] \bigg), \numberthis \label{eqn-del_t-m-l-Rm} \\[5pt]
    \del_t \big( \nabla^m \nablabar^l T \big) &= \sum_{i+j>0} \sum_{i=0}^m \sum_{j=0}^l (\nabla^{m-i} \nablabar^{l-j} T) * \big( \nabla^i \nablabar^j (\del_t g) \big) \\
    &\quad + \nabla^m \nablabar^l \bigg( \Big( \frac{1}{2 |\Omega|_\omega} \Big) \cdot \Big[ \frac{1}{2} \Delta_R T + K_1 + \alpha' \big( \nabla (\Rm * \Rm) + \nabla (F * F) + K_2 \big) \Big] \bigg), \numberthis \label{eqn-del_t-m-l-T} \\[5pt]
    \del_t \big( \nabla^m \nablabar^l F \big) &= \sum_{i+j>0} \sum_{i=0}^m \sum_{j=0}^l (\nabla^{m-i} \nablabar^{l-j} F) * \big( \nabla^i \nablabar^j (\del_t H) \big) \\
    &\quad + \nabla^m \nablabar^l \bigg( \Big( \frac{1}{2 |\Omega|_\omega} \Big) \cdot \Big[ \frac{1}{2} \Delta_R F + \Rm * F + F * F + \nablabar (T * F) + \nabla (\overline{T} * F) \Big] \bigg). \numberthis \label{eqn-del_t-m-l-F}
\end{align*}

\section{Shi-Type Estimates}
\label{sect-Shi-ests}

We first state some assumptions that will be used throughout the rest of this paper. These will largely coincide with those from \cite{Sua25}, with slight changes in grouping terms and the addition of terms involving $F$.

We suppose for some $k \geq 1$ that there exist positive constants $B_{min}, B_{max}, C_{-1}, C_0, \ldots, C_{k-1}$ such that
\begin{gather}
    B_{min} \leq \Big( \frac{1}{2 |\Omega|_\omega} \Big) \leq B_{max}, \label{eqn-ass-1} \\[5pt]
    |T|, |\overline{T}| \leq C_{-1}, \qquad |\Rm|, |DT|, |D\overline{T}|, |F| \leq C_0, \label{eqn-ass-2} \\[5pt]
    |D^q \Rm|, |D^{q+1} T|, |D^{q+1} \overline{T}|, |D^q F| \leq C_q \text{ for } 1 \leq q \leq k-1, \label{eqn-ass-3}
\end{gather}
along the anomaly flow \eqref{eqn-anomaly-flow} on $t \in [0,\tau)$, where
\begin{align*}
    |D^q A|^2 = \sum_{m+l = q} |\nabla^m \nablabar^l A|^2.
\end{align*}

We will make heavy use of the test functions $G_q$ for $q \in \bb{N}_0$ given by
\begin{align*}
    G_q = |D^q \Rm|^2 + |D^{q+1} T|^2 + |D^q F|^2.
\end{align*}

\begin{rmk}
\label{rmk-unif-bdd}
    We note that the assumptions \eqref{eqn-ass-1} - \eqref{eqn-ass-2} tell us that the evolving metrics $g$ and $H$ are uniformly bounded above and below by the initial metrics. In turn, we get that the volume $\int_X 1\, \vol$ of $X$ and also that the Sobolev constants for $X$ are uniformly bounded.
\end{rmk}

\subsection{The \texorpdfstring{$k \geq 2$}{k >= 2} Case}
\label{subsect-k>=2}

\begin{propn}
\label{propn-G_k}
    Let $k \geq 2$, $p \geq 3$. Under the assumptions \eqref{eqn-ass-1} - \eqref{eqn-ass-3}, if
    \begin{align*}
        \alpha' < \frac{B_{min}}{8 a_0 B_{max} C_0 p},
    \end{align*}
    then
    \begin{align*}
        |G_k|_{L^p}, |D^k \Rm|_{L^{2p}}, |D^{k+1} T|_{L^{2p}} \text{ and } |D^k F|_{L^{2p}}
    \end{align*}
    are uniformly bounded along the flow.
\end{propn} 

\begin{proof}
    The proof largely mirrors that of the one in \cite{Sua25}. As such, we shall skip most of the computations and only show the major steps.

    One can check that when accounting for our choice of $\Phi = \tr (F \wedge F)$ that we get the following pointwise estimates for the evolution of $|D^k \Rm|^2$ and $|D^{k+1} T|^2$ for $0 < \eps < 1$:
    \begin{align*}
        \del_t |D^k \Rm|^2 &\leq \Big( \frac{1}{2 |\Omega|_\omega} \Big) \cdot \Big( \frac{1}{2} \Delta_R \big( |D^k \Rm|^2 \big) \Big) \\
        &\qquad + \sum_{m+l=k} 2 \alpha' \Re \Big( \nablabar_{\overline{j}} \Binner{\Big( \frac{1}{2 |\Omega|_\omega} \Big) \cdot \nabla^{m+1} \nablabar^l (\Rm * \Rm)}{\nabla^m \nablabar^l \Rm}^{\overline{j}} \Big) \\
        &\qquad + \sum_{m+l=k} 2 \alpha' \Re \Big( \nablabar_{\overline{j}} \Binner{\Big( \frac{1}{2 |\Omega|_\omega} \Big) \cdot \nabla^{m+1} \nablabar^l (F * F)}{\nabla^m \nablabar^l \Rm}^{\overline{j}} \Big) \\
        &\qquad + C \eps^{-1} \Big( 1 + |D^k \Rm|^2 + |D^{k+1} T|^2 + |D^k F|^2 \Big) \\
        &\qquad + \big[ C \eps + 6 a_0 B_{max} C_0 \alpha' - B_{min} \big] |D^{k+1} \Rm|^2 + C \eps |D^{k+2} T|^2 + \big[ C \eps + 2 a_0 B_{max} C_0 \alpha' \big] |D^{k+1} F|^2, \\[5pt]
        \del_t |D^{k+1} T|^2 &\leq \Big( \frac{1}{2 |\Omega|_\omega} \Big) \cdot \Big( \frac{1}{2} \Delta_R \big( |D^{k+1} \Rm|^2 \big) \Big) \\
        &\qquad + \sum_{m' + l' = k+1} 2 \alpha' \Re \Big( \nabla_i \Binner{\Big( \frac{1}{2 |\Omega|_\omega} \Big) \cdot \nabla^{m'} \nablabar^{l'+1} (\Rm * \Rm)}{\nabla^{m'} \nablabar^{l'} T}^i \Big) \\
        &\qquad + \sum_{m' + l' = k+1} 2 \alpha' \Re \Big( \nabla_i \Binner{\Big( \frac{1}{2 |\Omega|_\omega} \Big) \cdot \nabla^{m'} \nablabar^{l'+1} (F * F)}{\nabla^{m'} \nablabar^{l'} T}^i \Big) \\
        &\qquad + C \eps^{-1} \Big( 1 + |D^k \Rm|^2 + |D^{k+1} T|^2 + |D^k F|^2 \Big) \\
        &\qquad + \big[ C \eps + 2 a_0 B_{max} C_0 \alpha' \big] |D^{k+1} \Rm|^2 + \big[ C \eps + 4 a_0 B_{max} C_0 \alpha' - B_{min} \big] |D^{k+2} T|^2 \\
        &\qquad \qquad + \big[ C \eps + 2 a_0 B_{max} C_0 \alpha' \big] |D^{k+1} F|^2,
    \end{align*}
    where $a_0 \geq 1$ is a uniform constant independent of k arising from the linear combinations obscured by the $*$ contraction notation in \eqref{eqn-del_t-g} - \eqref{eqn-del_t-T}.

    A similar treatment for the evolution of $|D^k F|^2$ using the schematic expression \eqref{eqn-del_t-m-l-F} yields the estimate
    \begin{align*}
        \del_t |D^k F|^2 &\leq \Big( \frac{1}{2 |\Omega|_\omega} \Big) \cdot \Big( \frac{1}{2} \Delta_R \big( |D^k F|^2 \big) \Big) \\
        &\qquad + C \eps^{-1} \Big( 1 + |D^k \Rm|^2 + |D^{k+1} T|^2 + |D^k F|^2 \Big) \\
        &\qquad + \big[ C \eps - B_{min} \big] |D^{k+1} F|^2.
    \end{align*}

    Consider the test functions
    \begin{align*}
        G_q = |D^q \Rm|^2 + |D^{q+1} T|^2 + |D^q F|^2.
    \end{align*}

    The above pointwise estimates combine to give
    \begin{align*}
        \del_t G_k &\leq \Big( \frac{1}{2 |\Omega|_\omega} \Big) \cdot \Big( \frac{1}{2} \Delta_R \big( G_k \big) \Big) \\
        &\qquad + \sum_{m+l=k} 2 \alpha' \Re \Big( \nablabar_{\overline{j}} \Binner{\Big( \frac{1}{2 |\Omega|_\omega} \Big) \cdot \nabla^{m+1} \nablabar^l (\Rm * \Rm)}{\nabla^m \nablabar^l \Rm}^{\overline{j}} \Big) \\
        &\qquad + \sum_{m+l=k} 2 \alpha' \Re \Big( \nablabar_{\overline{j}} \Binner{\Big( \frac{1}{2 |\Omega|_\omega} \Big) \cdot \nabla^{m+1} \nablabar^l (F * F)}{\nabla^m \nablabar^l \Rm}^{\overline{j}} \Big) \\
        &\qquad + \sum_{m' + l' = k+1} 2 \alpha' \Re \Big( \nabla_i \Binner{\Big( \frac{1}{2 |\Omega|_\omega} \Big) \cdot \nabla^{m'} \nablabar^{l'+1} (\Rm * \Rm)}{\nabla^{m'} \nablabar^{l'} T}^i \Big) \\
        &\qquad + \sum_{m' + l' = k+1} 2 \alpha' \Re \Big( \nabla_i \Binner{\Big( \frac{1}{2 |\Omega|_\omega} \Big) \cdot \nabla^{m'} \nablabar^{l'+1} (F * F)}{\nabla^{m'} \nablabar^{l'} T}^i \Big) \\
        &\qquad + C \eps^{-1} \Big( 1 + G_k \Big) + \Big[ C \eps + 8 a_0 B_{max} C_0 \alpha' - B_{min} \Big] G_{k+1}.
    \end{align*}

    Since the volume of $X$ is uniformly bounded by Remark \ref{rmk-unif-bdd}, we see that for $p \geq 3$ that
    \begin{align*}
        \del_t \bigg( \int_X G_k^p\, \vol \bigg) &\leq \int_X p G_k^{p-1} \cdot \Big( \frac{1}{2 |\Omega|_\omega} \Big) \cdot \Big( \frac{1}{2} \Delta_R \big( G_k \big) \Big)\, \vol \\
        &\qquad + \sum_{m+l=k} 2 \alpha' \Re \bigg( \int_X p G_k^{p-1} \cdot \nablabar_{\overline{j}} \Binner{\Big( \frac{1}{2 |\Omega|_\omega} \Big) \cdot \nabla^{m+1} \nablabar^l (\Rm * \Rm)}{\nabla^m \nablabar^l \Rm}^{\overline{j}}\, \vol \bigg) \\
        &\qquad + \sum_{m+l=k} 2 \alpha' \Re \bigg( \int_X p G_k^{p-1} \cdot \nablabar_{\overline{j}} \Binner{\Big( \frac{1}{2 |\Omega|_\omega} \Big) \cdot \nabla^{m+1} \nablabar^l (F * F)}{\nabla^m \nablabar^l \Rm}^{\overline{j}}\, \vol \bigg) \\
        &\qquad + \sum_{m' + l' = k+1} 2 \alpha' \Re \bigg( \int_X p G_k^{p-1} \cdot \nabla_i \Binner{\Big( \frac{1}{2 |\Omega|_\omega} \Big) \cdot \nabla^{m'} \nablabar^{l'+1} (\Rm * \Rm)}{\nabla^{m'} \nablabar^{l'} T}^i\, \vol \bigg) \\
        &\qquad + \sum_{m' + l' = k+1} 2 \alpha' \Re \bigg( \int_X p G_k^{p-1} \cdot \nabla_i \Binner{\Big( \frac{1}{2 |\Omega|_\omega} \Big) \cdot \nabla^{m'} \nablabar^{l'+1} (F * F)}{\nabla^{m'} \nablabar^{l'} T}^i\, \vol \bigg) \\
        &\qquad + C \eps^{-1} \int_X G_k^{p-1} \cdot \big( 1 + G_k \big)\, \vol + \big[ C \eps + 8 a_0 B_{max} C_0 \alpha' - B_{min} \big] \int_X p G_k^{p-1} \cdot G_{k+1}\, \vol.
    \end{align*}

    From here, computations as in \cite{Sua25} then show that
    \begin{align*}
        \int_X p G_k^{p-1} \cdot \Big( \frac{1}{2 |\Omega|_\omega} \Big) \cdot \Big( \frac{1}{2} \Delta_R \big( G_k \big) \Big)\, \vol &\leq C \eps^{-1} \int_X G_k^p \, \vol + \big[ C \eps - B_{min} p (p-1) \big] \int_X G_k^{p-2} \cdot |\nabla G_k|^2 \, \vol,
    \end{align*}
    and also that each of the following four lines (that is, the ones involving summations) are bounded above by
    \begin{align*}
        &C \eps^{-1} \int_X \big( 1 + G_k^p \big) \, \vol + \big[ C \eps + 2 a_0 B_{max} C_0 \alpha' p (p-1) \big] \int_X G_k^{p-1} \cdot G_{k+1} \, \vol \\
        &\qquad + \big[ C \eps + 2 a_0 B_{max} C_0 \alpha' p (p-1) \big] \int_X G_k^{p-2} \cdot |\nablabar G_k|^2 \, \vol.
    \end{align*}

    Thus we have
    \begin{align*}
        \del_t \bigg( \int_X G_k^p \, \vol \bigg) &\leq C \eps^{-1} \int_X \big( 1 + G_k^p \big) \, \vol + \big[ C \eps + 8 a_0 B_{max} C_0 \alpha' p^2 - B_{min} p \big] \int_X G_k^{p-1} \cdot G_{k+1} \, \vol \\
        &\qquad + \big[ C \eps + 8 a_0 B_{max} C_0 \alpha' p (p-1) - B_{min} p (p-1) \big] \int_X G_k^{p-2} \cdot|\nablabar G_k|^2 \, \vol.
    \end{align*}

    Thus if
    \begin{align*}
        \alpha' < \frac{B_{min}}{8 a_0 B_{max} C_0 p},
    \end{align*}
    we can choose $\eps$ small enough such that
    \begin{align*}
        \del_t \bigg( \int_X G_k^p \, \vol \bigg) \leq C + C \int_X G_k^p \, \vol.
    \end{align*}

    By Gr\"{o}nwall's inequality, we conclude that there exists a positive constant $\Lambda_p = \Lambda_p(k, \alpha')$ such that
    \begin{align*}
        \int_X G_k^p(t) \, \vol(t) < \bigg( 1 + \int_X G_k^p(0) \, \vol(0) \bigg) e^{\Lambda_p \tau}
    \end{align*}
    for all $t \in [0,\tau)$, from which the result follows.
\end{proof}

\begin{rmk}
    We note that we chose $p \geq 3$ in the above proof to avoid potentially dividing by zero since the omitted steps using Young's inequality involves terms such as $G_k^{p-3}$. Uniform boundedness of the $L^p$- and $L^{2p}$-norms of the respective quantities involved for $1 \leq p < 3$ can ultimately be recovered via H\"{o}lder's inequality since $X$ is compact. 
\end{rmk}

We can also obtain a similar bound on $L^p$-norms of higher-order derivatives in a similar manner:

\begin{propn}
\label{propn-G_k+1}
    Let $k \geq 2$, $p \geq 3$. Under the assumptions \eqref{eqn-ass-1} - \eqref{eqn-ass-3}, if
    \begin{align*}
        \alpha' < \frac{B_{min}}{8 a_0 B_{max} C_0 p}, \qquad \bigg( \text{ or } \alpha' < \frac{B_{min}}{16 a_0 B_{max} C_0 p} \text{ instead if $k = 2$} \bigg)
    \end{align*}
    then
    \begin{align*}
        |G_{k+1}|_{L^p}, |D^{k+1} \Rm|_{L^{2p}}, |D^{k+2} T|_{L^{2p}} \text{ and } |D^{k+1} F|_{L^{2p}}
    \end{align*}
    are uniformly bounded along the flow.
\end{propn}

\begin{proof}
    The proof is analogous to the one from before, and is based on the pointwise estimate
    \begin{align*}
        \del_t G_{k+1} &\leq \Big( \frac{1}{2 |\Omega|_\omega} \Big) \cdot \Big( \frac{1}{2} \Delta_R \big( G_{k+1} \big) \Big) \\
        &\qquad + \sum_{m+l=k+1} 2 \alpha' \Re \Big( \nablabar_{\overline{j}} \Binner{\Big( \frac{1}{2 |\Omega|_\omega} \Big) \cdot \nabla^{m+1} \nablabar^l (\Rm * \Rm)}{\nabla^m \nablabar^l \Rm}^{\overline{j}} \Big) \\
        &\qquad + \sum_{m+l=k+1} 2 \alpha' \Re \Big( \nablabar_{\overline{j}} \Binner{\Big( \frac{1}{2 |\Omega|_\omega} \Big) \cdot \nabla^{m+1} \nablabar^l (F * F)}{\nabla^m \nablabar^l \Rm}^{\overline{j}} \Big) \\
        &\qquad + \sum_{m' + l' = k+2} 2 \alpha' \Re \Big( \nabla_i \Binner{\Big( \frac{1}{2 |\Omega|_\omega} \Big) \cdot \nabla^{m'} \nablabar^{l'+1} (\Rm * \Rm)}{\nabla^{m'} \nablabar^{l'} T}^i \Big) \\
        &\qquad + \sum_{m' + l' = k+2} 2 \alpha' \Re \Big( \nabla_i \Binner{\Big( \frac{1}{2 |\Omega|_\omega} \Big) \cdot \nabla^{m'} \nablabar^{l'+1} (F * F)}{\nabla^{m'} \nablabar^{l'} T}^i \Big) \\
        &\qquad + C \eps^{-1} \Big( 1 + G_k + G_{k+1} \Big) + \Big[ C \eps + 8 a_0 B_{max} C_0 \alpha' - B_{min} \Big] G_{k+2}
    \end{align*}
    when $k \geq 3$. Proceeding from here (using $L^p$-bounds followed by Gr\"{o}nwall's inequality) yields 
    \begin{align*}
        \del_t \bigg( \int_X G_{k+1}^p \, \vol \bigg) \leq C + C \int_X G_k^p \, \vol + C \int_X G_{k+1}^p \, \vol
    \end{align*}
    when
    \begin{align*}
        \alpha' < \frac{B_{min}}{8 a_0 B_{max} C_0 p}.
    \end{align*}
    The result follows as long as $\int_X G_k^p \, \vol$ is uniformly bounded, which holds by Proposition \ref{propn-G_k} under the same assumption on $\alpha'$.

    When $k = 2$, the estimate instead becomes
    \begin{align*}
        \del_t G_{3} &\leq \Big( \frac{1}{2 |\Omega|_\omega} \Big) \cdot \Big( \frac{1}{2} \Delta_R \big( G_{3} \big) \Big) \\
        &\qquad + \sum_{m+l=3} 2 \alpha' \Re \Big( \nablabar_{\overline{j}} \Binner{\Big( \frac{1}{2 |\Omega|_\omega} \Big) \cdot \nabla^{m+1} \nablabar^l (\Rm * \Rm)}{\nabla^m \nablabar^l \Rm}^{\overline{j}} \Big) \\
        &\qquad + \sum_{m+l=3} 2 \alpha' \Re \Big( \nablabar_{\overline{j}} \Binner{\Big( \frac{1}{2 |\Omega|_\omega} \Big) \cdot \nabla^{m+1} \nablabar^l (F * F)}{\nabla^m \nablabar^l \Rm}^{\overline{j}} \Big) \\
        &\qquad + \sum_{m' + l' = 4} 2 \alpha' \Re \Big( \nabla_i \Binner{\Big( \frac{1}{2 |\Omega|_\omega} \Big) \cdot \nabla^{m'} \nablabar^{l'+1} (\Rm * \Rm)}{\nabla^{m'} \nablabar^{l'} T}^i \Big) \\
        &\qquad + \sum_{m' + l' = 4} 2 \alpha' \Re \Big( \nabla_i \Binner{\Big( \frac{1}{2 |\Omega|_\omega} \Big) \cdot \nabla^{m'} \nablabar^{l'+1} (F * F)}{\nabla^{m'} \nablabar^{l'} T}^i \Big) \\
        &\qquad + C \eps^{-1} \Big( 1 + G_2^2 + G_{3} \Big) + \Big[ C \eps + 8 a_0 B_{max} C_0 \alpha' - B_{min} \Big] G_{4}.
    \end{align*}
    In this case we eventually get
    \begin{align*}
        \del_t \bigg( \int_X G_3^p \, \vol \bigg) \leq C + C \int_X G_2^{2p} \, \vol + C \int_X G_3^p \, \vol,
    \end{align*}
    when
    \begin{align*}
        \alpha' < \frac{B_{min}}{8 a_0 B_{max} C_0 p}.
    \end{align*}
    Since we also require that $\int_X G_2^{2p} \, \vol$ be bounded we require the stronger bound (replacing $p$ with $2p$) from Proposition \ref{propn-G_k}.
\end{proof}

\subsection{The \texorpdfstring{$k = 1$}{k = 1} Case and a Revised Smallness Condition}
\label{subsect-k=1}

We now prove the base $k = 1$ case and obtain a modified bound which agrees with the $k \geq 2$ case in terms of physical units. In particular, we obtain a bound of the form
\begin{align*}
    \alpha' \cdot |\Rm| < \Pi
\end{align*}
for some dimensionless constant $\Pi$, which is in line with the evolution equation
\begin{align*}
    \del_t g_{\overline{p} q} = \Big( \frac{1}{2 |\Omega|_\omega} \Big) \Big[ - \wtilde{R}_{\overline{p} q} + g^{\alpha \overline{\beta}} g^{s \overline{r}} T_{\overline{\beta} s q} \overline{T}_{\alpha \overline{r} \overline{p}} - \alpha' g^{s \overline{r}} \big( \tensor{R}{_[_{\overline{p}}_s^a_b} \tensor{R}{_{\overline{r}}_q_]^b_a} - \tensor{F}{_[_{\overline{p}}_s^\alpha_\beta} \tensor{F}{_{\overline{r}}_q_]^\beta_\alpha} \big) \Big]
\end{align*}
from \cite{PPZ18b}.

As before, the main idea is the same as that of \cite{Sua25}, however we use Young's inequality to separate terms in a slightly different fashion, keeping powers of $\alpha'$ and $C_0$ ``balanced".

\begin{propn}
\label{propn-G_k-k=1}
    Let $k = 1$ and $p \geq 3$. Under the assumptions \eqref{eqn-ass-1} - \eqref{eqn-ass-3}, there exists a positive constant $M$ such that if
    \begin{align*}
        \alpha' < M \cdot \frac{B_{min}^2}{a_0 B_{max}^2 C_0 p},
    \end{align*}
    then
    \begin{align*}
        |D \Rm|_{L^{2p}}, |D^2 T|_{L^{2p}} \text{ and } |D F|_{L^{2p}}
    \end{align*}
    are uniformly bounded along the flow.
\end{propn}

\begin{proof}
    One can again check that we have the pointwise estimates
    \begin{align*}
        \del_t G_0 &\leq \Big( \frac{1}{2 |\Omega|_\omega} \Big) \cdot \Big( \frac{1}{2} \Delta_R \big( G_0 \big) \Big) \\
        &\qquad + 2 \alpha' \Re \Big( \nablabar_{\overline{j}} \Binner{\Big( \frac{1}{2 |\Omega|_\omega} \Big) \nabla (\Rm * \Rm)}{\Rm}^{\overline{j}} \Big) \\
        &\qquad + 2 \alpha' \Re \Big( \nablabar_{\overline{j}} \Binner{\Big( \frac{1}{2 |\Omega|_\omega} \Big) \nabla (F * F)}{\Rm}^{\overline{j}} \Big) \\
        &\qquad + \sum_{m' + l' = 1} 2 \alpha' \Re \Big( \nabla_i \Binner{\Big( \frac{1}{2 |\Omega|_\omega} \nabla^{m'} \nablabar^{l'} (\Rm * \Rm)}{\nabla^{m'} \nablabar^{l'} T}^i \Big) \\
        &\qquad + \sum_{m' + l' = 1} 2 \alpha' \Re \Big( \nabla_i \Binner{\Big( \frac{1}{2 |\Omega|_\omega} \nabla^{m'} \nablabar^{l'} (F * F)}{\nabla^{m'} \nablabar^{l'} T}^i \Big) \\
        &\qquad + C \eps^{-1} 1 + \big[ C \eps + 8 a_0 B_{max} C_0 \alpha' - B_{min} \big] G_1,
    \end{align*}
    and
    \begin{align*}
        \del_t G_1 &\leq \Big( \frac{1}{2 |\Omega|_\omega} \Big) \cdot \Big( \frac{1}{2} \Delta_R \big( G_1 \big) \Big) \\
        &\qquad + \sum_{m+l = 1} 2 \alpha' \Re \Big( \nablabar_{\overline{j}} \Binner{\Big( \frac{1}{2 |\Omega|_\omega} \Big) \nabla^{m+1} \nablabar^l (\Rm * \Rm)}{\nabla^m \nablabar^l \Rm}^{\overline{j}} \Big) \\
        &\qquad + \sum_{m+l = 1} 2 \alpha' \Re \Big( \nablabar_{\overline{j}} \Binner{\Big( \frac{1}{2 |\Omega|_\omega} \Big) \nabla^{m+1} \nablabar^l (F * F)}{\nabla^m \nablabar^l \Rm}^{\overline{j}} \Big) \\
        &\qquad + \sum_{m' + l' = 2} 2 \alpha' \Re \Big( \nabla_i \Binner{\Big( \frac{1}{2 |\Omega|_\omega} \nabla^{m'} \nablabar^{l'} (\Rm * \Rm)}{\nabla^{m'} \nablabar^{l'} T}^i \Big) \\
        &\qquad + \sum_{m' + l' = 2} 2 \alpha' \Re \Big( \nabla_i \Binner{\Big( \frac{1}{2 |\Omega|_\omega} \nabla^{m'} \nablabar^{l'} (F * F)}{\nabla^{m'} \nablabar^{l'} T}^i \Big) \\
        &\qquad + C \eps^{-1} \big( 1 + G_1 \big) + \big[ C \eps + 10 a_0 B_{max} C_0^{-1} \alpha' \big] G_1^2 + \big[ C \eps + 12 a_0 B_{max} C_0 \alpha' - B_{min} \big] G_2.
    \end{align*}

    Here, the coefficients for $G_1^2$ and $G_2$ were obtained using Young's inequality on the highest-order terms appearing in the computation for $\del_t G_1$ (\cf the computation for Term $\mathbf{(VIII)}$ in \cite{Sua25}):
    {\allowdisplaybreaks
    \begin{align*}
        &\sum_{m+l=1} 2 \alpha' \Big( \frac{1}{2 |\Omega|_\omega} \Big) \cdot |\nabla^{m+1} \nablabar^l (\Rm * \Rm)| \cdot |\nabla^{m+1} \nablabar^l \Rm| \\
        &\qquad + \sum_{m+l=1} 2 \alpha' \Big( \frac{1}{2 |\Omega|_\omega} \Big) \cdot |\nabla^{m+1} \nablabar^l (F * F)| \cdot |\nabla^{m+1} \nablabar^l \Rm| \\
        &\qquad + \sum_{m'+l'=2} 2 \alpha' \Big( \frac{1}{2 |\Omega|_\omega} \Big) \cdot |\nabla^{m'} \nablabar^{l'} (\Rm * \Rm)| \cdot |\nabla^{m'} \nablabar^{l'+1} T| \\
        &\qquad + \sum_{m'+l'=2} 2 \alpha' \Big( \frac{1}{2 |\Omega|_\omega} \Big) \cdot |\nabla^{m'} \nablabar^{l} (F * F)| \cdot |\nabla^{m'} \nablabar^{l'+1} T| \\
        &= \sum_{m+l=1} \Big( 4 a_0 B_{max} C_0 \alpha' |\nabla^{m+1} \nablabar^l \Rm| \cdot |\nabla^{m+1} \nablabar^l \Rm| + 4 a_0 B_{max} \alpha' |D \Rm|^2 \cdot |\nabla^{m+1} \nablabar^l \Rm| \Big) \\
        &\qquad + \sum_{m+l=1} \Big( 4 a_0 B_{max} C_0 \alpha' |\nabla^{m+1} \nablabar^l F| \cdot |\nabla^{m+1} \nablabar^l \Rm| + 4 a_0 B_{max} \alpha' |D F|^2 \cdot |\nabla^{m+1} \nablabar^l \Rm| \Big) \\
        &\qquad + \sum_{m'+l'=2} \Big( 4 a_0 B_{max} C_0 \alpha' |\nabla^{m'} \nablabar^{l'} \Rm| \cdot |\nabla^{m'} \nablabar^{l'+1} T| + 4 a_0 B_{max} \alpha' |D \Rm|^2 \cdot |\nabla^{m'} \nablabar^{l'+1} T| \Big) \\
        &\qquad + \sum_{m'+l'=2} \Big( 4 a_0 B_{max} C_0 \alpha' |\nabla^{m'} \nablabar^{l'} F| \cdot |\nabla^{m'} \nablabar^{l'+1} T| + 4 a_0 B_{max} \alpha' |D F|^2 \cdot |\nabla^{m'} \nablabar^{l'+1} T| \Big) \\
        &\leq 4 a_0 B_{max} C_0 \alpha' |D^2 \Rm|^2 + 4 a_0 B_{max} C_0^{-1} \alpha' |D \Rm|^4 + 2 a_0 B_{max} C_0 \alpha' |D^2 \Rm|^2 \\*
        &\qquad + 2 a_0 B_{max} C_0 \alpha' |D^2 F|^2 + 2 a_0 B_{max} C_0 \alpha' |D^2 \Rm|^2 + 4 a_0 B_{max} C_0^{-1} \alpha' |DF|^4 + 2 a_0 B_{max} C_0 \alpha' |D^2 \Rm|^2 \\
        &\qquad + 2 a_0 B_{max} C_0 \alpha' |D^2 \Rm|^2 + 2 a_0 B_{max} C_0 \alpha' |D^3 T|^2 + 6 a_0 B_{max} C_0^{-1} \alpha' |D \Rm|^4 + 2 a_0 B_{max} C_0 \alpha' |D^3 T|^2 \\
        &\qquad + 2 a_0 B_{max} C_0 \alpha' |D^2 F|^2 + 2 a_0 B_{max} C_0 \alpha' |D^3 T|^2 + 6 a_0 B_{max} C_0^{-1} \alpha' |D F|^4 + 2 a_0 B_{max} C_0 \alpha' |D^3 T|^2 \\
        &\leq 10 a_0 B_{max} C_0^{-1} \alpha' G_1^2 + 12 a_0 B_{max} C_0 \alpha' G_2.
    \end{align*}}
    
    Consider the test function
    \begin{align*}
        G = \big[ (\alpha')^2 G_0 + \mu \big] \cdot G_1
    \end{align*}
    for some positive dimensionless constant $\mu$ to be determined later. We note that this test function is consistent with the idea that $(\alpha')^2 G_0$ should be dimensionless.

    Using the above and the fact that $G_0 = |\Rm|^2 + |DT|^2 + |F|^2 \leq 3 C_0$ by \eqref{eqn-ass-2}, we see that
    \begin{align*}
        \del_t G &\leq \Big( \frac{1}{2 |\Omega|_\omega} \Big) \cdot \Big( \frac{1}{2} \Delta_R \big( G \big) \Big) - \Big( \frac{1}{2 |\Omega|_\omega} \Big) \cdot (\alpha')^2 \Re \inner{\nabla G_0}{\nabla G_1} \\
        &\qquad + C \eps^{-1} \big( 1 + G \big) + 2 \Re \mathbf{(A)} + 2 \Re \mathbf{(B)} \\
        &\qquad + \big[ C \eps + 38 a_0 B_{max} C_0 (\alpha')^3 + 10 a_0 B_{max} C_0^{-1} \alpha' \mu - B_{min} (\alpha')^2 \big] G_1^2 \\
        &\qquad + \big[ C \eps + 36 a_0 B_{max} C_0^3 (\alpha')^3 + 12 a_0 B_{max} C_0 \alpha' \mu - B_{min} \mu \big] G_2
    \end{align*}
    where
    \begin{align*}
        \mathbf{(A)} &= (\alpha')^3 G_1 \cdot \Big( \nablabar_{\overline{j}} \Binner{\Big( \frac{1}{2 |\Omega|_\omega} \Big) \nabla (\Rm * \Rm)}{\Rm}^{\overline{j}} \Big) \\
        &\qquad + (\alpha')^3 G_1 \cdot \Big( \nablabar_{\overline{j}} \Binner{\Big( \frac{1}{2 |\Omega|_\omega} \Big) \nabla (F * F)}{\Rm}^{\overline{j}} \Big) \\
        &\qquad + \sum_{m' + l' = 1} (\alpha')^3 G_1 \cdot \Big( \nabla_i \Binner{\Big( \frac{1}{2 |\Omega|_\omega} \nabla^{m'} \nablabar^{l'} (\Rm * \Rm)}{\nabla^{m'} \nablabar^{l'} T}^i \Big) \\
        &\qquad + \sum_{m' + l' = 1} (\alpha')^3 G_1 \cdot \Big( \nabla_i \Binner{\Big( \frac{1}{2 |\Omega|_\omega} \nabla^{m'} \nablabar^{l'} (F * F)}{\nabla^{m'} \nablabar^{l'} T}^i \Big),
    \end{align*}
    and
    \begin{align*}
        \mathbf{(B)} &= \sum_{m+l = 1} \alpha' \big[ (\alpha')^2 G_0 + \mu \big] \cdot \Big( \nablabar_{\overline{j}} \Binner{\Big( \frac{1}{2 |\Omega|_\omega} \Big) \nabla^{m+1} \nablabar^l (\Rm * \Rm)}{\nabla^m \nablabar^l \Rm}^{\overline{j}} \Big) \\
        &\qquad + \sum_{m+l = 1} \alpha' \big[ (\alpha')^2 G_0 + \mu \big] \cdot\Big( \nablabar_{\overline{j}} \Binner{\Big( \frac{1}{2 |\Omega|_\omega} \Big) \nabla^{m+1} \nablabar^l (F * F)}{\nabla^m \nablabar^l \Rm}^{\overline{j}} \Big) \\
        &\qquad + \sum_{m' + l' = 2} \alpha' \big[ (\alpha')^2 G_0 + \mu \big] \cdot\Big( \nabla_i \Binner{\Big( \frac{1}{2 |\Omega|_\omega} \nabla^{m'} \nablabar^{l'} (\Rm * \Rm)}{\nabla^{m'} \nablabar^{l'} T}^i \Big) \\
        &\qquad + \sum_{m' + l' = 2} \alpha' \big[ (\alpha')^2 G_0 + \mu \big] \cdot\Big( \nabla_i \Binner{\Big( \frac{1}{2 |\Omega|_\omega} \nabla^{m'} \nablabar^{l'} (F * F)}{\nabla^{m'} \nablabar^{l'} T}^i \Big).
    \end{align*}

    For $p \geq 3$, the $L^p$-version of this estimate becomes
    \begin{align*}
        \del_t \bigg( \int_X G^p \, \vol \bigg) &\leq \int_X p G^{p-1} \cdot \Big( \frac{1}{2 |\Omega|_\omega} \Big) \cdot \Big( \frac{1}{2} \Delta_R \big( G \big) \Big) \, \vol \\
        &\qquad + (\alpha')^2 \Re \bigg( \int_X p G^{p-1} \cdot \Big( \frac{1}{2 |\Omega|_\omega} \Big) \cdot \inner{\nabla G_0}{\nabla G_1} \, \vol \bigg) + C \eps^{-1} \int_X G^{p-1} \cdot \big( 1 + G \big) \, \vol \\
        &\qquad + 2 \Re \bigg( \int_X p G^{p-1} \cdot \mathbf{(A)} \, \vol \bigg)  + 2 \Re \bigg( \int_X p G^{p-1} \cdot \mathbf{(B)} \, \vol \bigg) \\
        &\qquad + \big[ C \eps + 38 a_0 B_{max} C_0 (\alpha')^3 + 10 a_0 B_{max} C_0^{-1} \alpha' \mu - B_{min} (\alpha')^2 \big] \int_X p G^{p-1} \cdot G_1^2 \, \vol \\
        &\qquad + \big[ C \eps + 36 a_0 B_{max} C_0^3 (\alpha')^3 + 12 a_0 B_{max} C_0 \alpha' \mu - B_{min} \mu \big] \int_X p G^{p-1} \cdot G_2 \, \vol.
    \end{align*}

    The Laplacian term is bounded above by
    \begin{align*}
        \int_X p G^{p-1} \cdot \Big( \frac{1}{2 |\Omega|_\omega} \Big) \cdot \Big( \frac{1}{2} \Delta_R \big( G \big) \Big) \, \vol \leq C \eps^{-1} \int_X G^p \, \vol + \big[ C \eps - B_{min} p (p-1) \big] \int_X G^{p-2} \cdot |\nabla G|^2 \, \vol.
    \end{align*}
    
    For the inner product term, we recall that
    \begin{align*}
        \big| \nabla |D\Rm|^2 \big| &\leq |\nabla \inner{\nabla \Rm}{\nabla \Rm}| + |\nabla \inner{\nablabar \Rm}{\nablabar \Rm}| \\
        &\leq |\nabla \Rm| \cdot |\nabla^2 \Rm| + |\nabla \Rm| \cdot |\nablabar \nabla \Rm| + |\nablabar \Rm| \cdot |\nabla \nablabar \Rm| + |\nablabar \Rm| \cdot |\nablabar^2 \Rm| \\
        &\leq C |\Rm| \cdot |D\Rm| + 4 |D \Rm| \cdot |D^2 \Rm|.
    \end{align*}

    Similar computations for $\big|\nabla |D^q \Rm|^2 \big|$, $\big|\nabla |D^{q+1} T|^2 \big|$, and $\big|\nabla |D^q F|^2 \big|$ show that
    \begin{align*}
        |\nabla G_0| &\leq Q_0 + 2 |\Rm| \cdot |D \Rm| + 4 |D T| \cdot |D^2 T| + 2 |F| \cdot |DF| \\[5pt]
        |\nabla G_1| &\leq Q_1 + 4 |D \Rm| \cdot |D^2 \Rm| + 6 |D^2 T| \cdot |D^3 T| + 4 |D F| \cdot |D^2 F|
    \end{align*}
    where the $Q_i$ are generic ``good" terms without too many derivatives which arise from commuting derivatives using the Ricci identity. In the following computations, the terms from the $Q_i$ can be absorbed into
    \begin{align*}
        C \eps^{-1} \int_X G^{p-1} \cdot \big( 1 + G \big) \vol + C \eps \int_X G^{p-1} \cdot G_1^2 \, \vol + C \eps \int_X G^{p-1} \cdot G_2 \, \vol
    \end{align*}
    when we apply Young's inequality.

    Proceeding with the computation we get
    \begin{align*}
        &(\alpha')^2 \Re \bigg( \int_X p G^{p-1} \cdot \Big( \frac{1}{2 |\Omega|_\omega} \Big) \cdot \inner{\nabla G_0}{\nabla G_1} \, \vol \bigg) \\
        &\leq C \eps^{-1} \int_X G^{p-1} \cdot \big( 1 + G \big) \vol + C \eps \int_X G^{p-1} \cdot G_1^2 \, \vol + C \eps \int_X G^{p-1} \cdot G_2 \, \vol \\
        &\quad + B_{max} (\alpha')^2 p \int_X G^{p-1} \cdot \big( 2 |\Rm| \cdot |D \Rm| + 4 |DT| \cdot |D^2 T| + 2 |F| \cdot |DF| \big) \\
        &\qquad \qquad \qquad \qquad \qquad \qquad \cdot \big( 4 |D \Rm| \cdot |D^2 \Rm| + 6 |D^2 T| \cdot |D^3 T| + 4 |DF| \cdot |D^2 F| \big) \, \vol \\
        &\leq C \eps^{-1} \int_X G^{p-1} \cdot \big( 1 + G \big) \vol \\
        &\qquad + \big[ C \eps + \frac{1}{2} B_{min} (\alpha')^2 p \big] \int_X G^{p-1} \cdot G_1^2 \, \vol + \big[ C \eps + 432 B_{max}^2 B_{min}^{-1} C_0^2 (\alpha')^2 p \big] \int_X G^{p-1} \cdot G_2 \, \vol.
    \end{align*}

    For the $\mathbf{(A)}$ term, using integration by parts and the divergence theorem, we see for example that
    \begin{align*}
        &2 \Re \bigg( \int_X p G^{p-1} \cdot (\alpha')^3 G_1 \cdot \Big( \nablabar_{\overline{j}} \Binner{\Big( \frac{1}{2 |\Omega|_\omega} \Big) \nabla (\Rm * \Rm)}{\Rm}^{\overline{j}} \Big) \, \vol \bigg) \\
        &\qquad + 2 \Re \bigg( \int_X p G^{p-1} \cdot (\alpha')^3 G_1 \cdot \Big( \nablabar_{\overline{j}} \Binner{\Big( \frac{1}{2 |\Omega|_\omega} \Big) \nabla (F * F)}{\Rm}^{\overline{j}} \Big) \, \vol \bigg) \\
        &\leq C \eps^{-1} \int_X G^p \, \vol + C \eps \int_X G^{p-1} \cdot G_1^2 \, \vol \\
        &\qquad + 4 a_0 B_{max} C_0^2 (\alpha')^3 p (p-1) \int_X G^{p-2} \cdot |\nabla G| \cdot G_1 \cdot |D\Rm| \, \vol \\
        &\qquad + 4 a_0 B_{max} C_0^2 (\alpha')^3 p (p-1) \int_X G^{p-2} \cdot |\nabla G| \cdot G_1 \cdot |DF| \, \vol \\
        &\qquad + 4 a_0 B_{max} C_0^2 (\alpha')^3 p \int_X G^{p-1} \cdot \big( Q_1 + 4 |D \Rm| \cdot |D^2 \Rm| + 6 |D^2 T| \cdot |D^3 T| + 4 |DF| \cdot |D^2 F| \big) \cdot |D \Rm| \, \vol \\
        &\qquad + 4 a_0 B_{max} C_0^2 (\alpha')^3 p \int_X G^{p-1} \cdot \big( Q_1 + 4 |D \Rm| \cdot |D^2 \Rm| + 6 |D^2 T| \cdot |D^3 T| + 4 |DF| \cdot |D^2 F| \big) \cdot |DF| \, \vol \\
        &\leq C \eps^{-1} \int_X G^p \, \vol + C \eps \int_X G^{p-1} \cdot G_1^2 \, \vol \\
        &\qquad + 4 a_0 B_{max} C_0^2 (\alpha')^2 p (p-1) \int_X G^{p-3} \cdot |\nabla G|^2 \cdot G_1 \, \vol + 2 a_0 B_{max} C_0^2 (\alpha')^4 p (p-1) \int_X G^{p-1} \cdot G_1^2 \, \vol \\
        &\qquad + 12 a_0 B_{max} C_0^2 (\alpha')^4 p \int_X G^{p-1} \cdot G_1^2 \, \vol + 24 a_0 B_{max} C_0^2 (\alpha')^2 p \int_X G^{p-1} \cdot G_2 \, \vol \\
        &= C \eps^{-1} \int_X G^p \, \vol \\
        &\qquad + \big[ C \eps + 2 a_0 B_{max} C_0^2 (\alpha')^4 p (p-1) + 12 a_0 B_{max} C_0^2 (\alpha')^4 p \big] \int_X G^{p-1} \cdot G_1^2 \, \vol \\
        &\qquad + 24 a_0 B_{max} C_0^2 (\alpha')^2 p \int_X G^{p-1} \cdot G_2 \, \vol + 4 a_0 B_{max} C_0^2 (\alpha')^2 \mu^{-1} p (p-1) \int_X G^{p-2} \cdot |\nabla G|^2 \, \vol.
    \end{align*}

    This can be repeated for the other parts of the $\mathbf{(A)}$ term.

    The $\mathbf{(B)}$ term needs to be treated more delicately. Analogous computations show that
    {\allowdisplaybreaks
    \begin{align*}
        &2 \Re \bigg( \int_X p G^{p-1} \cdot \alpha' \big[ (\alpha')^2 G_0 + \mu \big] \cdot \Big( \nablabar_{\overline{j}} \Binner{ \Big( \frac{1}{2 |\Omega|_\omega} \Big) \nabla^2 (\Rm * \Rm)}{\nabla \Rm}^{\overline{j}} \Big) \, \vol \bigg) \\
        &\qquad + 2 \Re \bigg( \int_X p G^{p-1} \cdot \alpha' \big[ (\alpha')^2 G_0 + \mu \big] \cdot \Big( \nablabar_{\overline{j}} \Binner{ \Big( \frac{1}{2 |\Omega|_\omega} \Big) \nabla^2 (F * F)}{\nabla \Rm}^{\overline{j}} \Big) \, \vol \bigg) \\
        &\leq C \eps^{-1} \int_X G^p \, \vol + C \eps \int_X G^{p-1} \cdot G_1^2 \, \vol + C \eps \int_X G^{p-1} \cdot G_2 \, \vol \\
        &\qquad + 2 a_0 B_{max} p (p-1) \int_X G^{p-2} \cdot |\nabla G| \cdot \alpha' \big[ (\alpha')^2 G_0 + \mu \big] \cdot \big( 2 |\Rm| \cdot |D^2 \Rm| + 2 |D \Rm|^2 \big) \cdot |D \Rm| \, \vol \\
        &\qquad + 2 a_0 B_{max} p (p-1) \int_X G^{p-2} \cdot |\nabla G| \cdot \alpha' \big[ (\alpha')^2 G_0 + \mu \big] \cdot \big( 2 |F| \cdot |D^2 F| + 2 |D F|^2 \big) \cdot |D \Rm| \, \vol \\
        &\qquad + 2 a_0 B_{max} (\alpha')^3 p \int_X G^{p-1} \cdot \big( Q_0 + 2 |\Rm| \cdot |D\Rm| + 4 |DT|^2 \cdot |D^2 T| + 2 |F| \cdot |DF| \big) \\
        &\qquad \qquad \qquad \qquad \qquad \qquad \qquad \cdot \big( 2 |\Rm| \cdot |D^2 \Rm| + 2 |D \Rm|^2 \big) \cdot |D \Rm| \, \vol \\
        &\qquad + 2 a_0 B_{max} (\alpha')^3 p \int_X G^{p-1} \cdot \big( Q_0 + 2 |\Rm| \cdot |D\Rm| + 4 |DT| \cdot |D^2 T| + 2 |F| \cdot |DF| \big) \\
        &\qquad \qquad \qquad \qquad \qquad \qquad \qquad \cdot \big( 2 |F| \cdot |D^2 F| + 2 |D F|^2 \big) \cdot |D \Rm| \, \vol \\
        &\leq C \eps^{-1} \int_X G^p \, \vol + C \eps \int_X G^{p-1} \cdot G_1^2 \, \vol + C \eps \int_X G^{p-1} \cdot G_2 \, \vol \\*
        &\qquad + \frac{1}{10} B_{min} p (p-1) \int_X G^{p-3} \cdot \big[ (\alpha')^2 G_0 + \mu \big] \cdot |D\Rm|^2 \cdot |\nabla G|^2 \, \vol \\
        &\qquad + 160 a_0^2 B_{max}^2 B_{min}^{-1} (\alpha')^2 p (p-1) \int_X G^{p-1} \cdot \big[ (\alpha')^2 G_0 + \mu \big] \cdot G_1^2 \, \vol \\
        &\qquad + 160 a_0^2 B_{max}^2 B_{min}^{-1} C_0^2 (\alpha')^2 p (p-1) \int_X G^{p-1} \cdot \big[ (\alpha')^2 G_0 + \mu \big] \cdot G_2 \, \vol \\
        &\qquad + 24 a_0 B_{max} C_0 (\alpha')^3 p \int_X G^{p-1} \cdot G_1^2 \, \vol + 16 a_0 B_{max} C_0^2 (\alpha')^4 p \int_X G^{p-1} \cdot G_1^2 \, \vol \\
        &\qquad + 16 a_0 B_{max} C_0^2 (\alpha')^2 p \int_X G^{p-1} \cdot G_2 \, \vol \\
        &\leq C \eps^{-1} \int_X G^p \, \vol \\
        &\qquad + \big[ C \eps + 160 a_0^2 B_{max}^2 B_{min}^{-1} (\alpha')^2 \mu p (p-1) + 480 a_0^2 B_{max}^2 B_{min}^{-1} C_0^2 (\alpha')^4 p (p-1) \\
        &\qquad \qquad + 24 a_0 B_{max} C_0 (\alpha')^3 p + 16 a_0 B_{max} C_0^2 (\alpha')^4 p \big] \int_X G^{p-1} \cdot G_1^2 \, \vol \\
        &\qquad + \big[ C \eps + 160 a_0^2 B_{max}^2 B_{min}^{-1} C_0^2 (\alpha')^2 \mu p (p-1) + 480 a_0^2 B_{max}^2 B_{min}^{-1} C_0^4 (\alpha')^4 p (p-1) \\
        &\qquad \qquad + 16 a_0 B_{max} C_0^2 (\alpha')^2 p \big] \int_X G^{p-1} \cdot G_2 \, \vol \\
        &\qquad + \big[ \frac{1}{10} B_{min} p (p-1) \big] \int_X G^{p-2} \cdot |\nabla G|^2 \, \vol.
    \end{align*}}

    This computation can also be repeated for the other parts of the $\mathbf{(B)}$ term.

    By combining everything thus far, we get
    \begin{align*}
        \del_t \bigg( \int_X G^p \, \vol \bigg) &\leq C \eps^{-1} \int_X \big( 1 + G^p \big) \, \vol \\
        &\qquad + \big[ C \eps + 800 a_0^2 B_{max}^2 B_{min}^{-1} (\alpha')^2 \mu p (p-1) + 2400 a_0^2 B_{max}^2 B_{min}^{-1} C_0^2 (\alpha')^4 p (p-1) \\
        &\qquad \qquad + 10 a_0 B_{max} C_0^{-1} \alpha' \mu p + 158 a_0 B_{max} C_0 (\alpha')^3 p + 6 a_0 B_{max} C_0^2 (\alpha')^4 p (p-1) \\
        &\qquad \qquad + 116 a_0 B_{max} C_0^2 (\alpha')^4 p - \frac{1}{2} B_{min} (\alpha')^2 p \big] \int_X G^{p-1} \cdot G_1^2 \, \vol \\
        &\qquad + \big[ C \eps + 800 a_0^2 B_{max}^2 B_{min}^{-1} C_0^2 (\alpha')^2 \mu p (p-1) + 2400 a_0^2 B_{max}^2 B_{min}^{-1} C_0^4 (\alpha')^4 p (p-1) \\
        &\qquad \qquad + 142 a_0 B_{max} C_0^2 (\alpha')^2 p + 36 a_0 B_{max} C_0^3 (\alpha')^3 p + 12 a_0 B_{max} C_0 \alpha' \mu p \\
        &\qquad \qquad + 432 B_{max}^2 B_{min}^{-1} C_0^2 (\alpha')^2 p - B_{min} \mu p \big] \int_X G^{p-1} \cdot G_2 \, \vol \\
        &\qquad + \big[ C \eps + 12 a_0 B_{max} C_0^2 (\alpha')^2 \mu^{-1} p (p-1) - \frac{1}{2} B_{min} p (p-1) \big] \int_X G^{p-1} \cdot |\nabla G|^2 \, \vol.
    \end{align*}

    Reading off the coefficients and cancelling some common factors, we see that if the inequalities
    \begin{gather*}
        800 a_0^2 B_{max}^2 B_{min}^{-1} \mu (p-1) + 2400 a_0^2 B_{max}^2 B_{min}^{-1} C_0^2 (\alpha')^2 (p-1) + 10 a_0 B_{max} C_0^{-1} (\alpha')^{-1} \mu \\
        \qquad + 158 a_0 B_{max} C_0 \alpha' + 6 a_0 B_{max} C_0^2 (\alpha')^2 (p-1) + 116 a_0 B_{max} C_0^2 (\alpha')^2 < \frac{1}{2} B_{min}, \numberthis \label{eqn-system-1} \\[5pt]
        800 a_0^2 B_{max}^2 B_{min}^{-1} C_0^2 (\alpha')^2 (p-1) + 2400 a_0^2 B_{max}^2 B_{min}^{-1} C_0^4 (\alpha')^4 \mu^{-1} (p-1) + 142 a_0 B_{max} C_0^2 (\alpha')^2 \mu^{-1} \\
        + 36 a_0 B_{max} C_0^3 (\alpha')^3 \mu^{-1} + 12 a_0 B_{max} C_0 \alpha' + 432 B_{max}^2 B_{min}^{-1} C_0^2 (\alpha')^2 \mu^{-1} < B_{min}, \numberthis \label{eqn-system-2} \\[5pt]
        12 a_0 B_{max} C_0^2 (\alpha')^2 \mu^{-1} < \frac{1}{2} B_{min} \numberthis \label{eqn-system-3}
    \end{gather*}
    are satisfied, then we are left with
    \begin{align*}
        \del_t \bigg( \int_X G^p \, \vol \bigg) \leq C + C \int_X G^p \, \vol.
    \end{align*}
    
    If we suppose that
    \begin{align*}
        \alpha' < M \cdot \frac{B_{min}^2}{a_0 B_{max}^2 C_0 p} \text{ and } \mu = N \cdot \frac{B_{min} C_0 \alpha'}{a_0 B_{max}} < MN \cdot \frac{B_{min}^3}{a_0^2 B_{max}^3 p}
    \end{align*}
    for some positive constants $M$ and $N$, the system \eqref{eqn-system-1} - \eqref{eqn-system-3} holds if
    \begin{gather*}
        158 M + 10 N + 800 MN + 2522 M^2 < \frac{1}{2}, \numberthis \label{eqn-system-b-1} \\[5pt]
        2400 \frac{M^3}{N} + 36 \frac{M^2}{N} + 574 \frac{M}{N} + 12 M + 800 M^2 < 1, \numberthis \label{eqn-system-b-2} \\[5pt]
        12 \frac{M}{N} < \frac{1}{2}, \numberthis \label{eqn-system-b-3}
    \end{gather*}
    where we have used that $B_{min} \leq B_{max}$, $p \geq 3$ and $a_0 \geq 1$.

    One can check that solutions to \eqref{eqn-system-b-1} -  \eqref{eqn-system-b-3} do exist. For instance, we can take $M = 8.413 \times 10^{-5}$ and $N = 4.834 \times 10^{-2}$. Proceeding using Gr\"{o}nwall's inequality as before shows that $|G|_{L^p}$ is uniformly bounded along the flow, which in turn proves the boundedness of $|D\Rm|_{L^{2p}}$, $|D^2 T|_{L^{2p}}$, and $|DF|_{L^{2p}}$.
\end{proof}

Applying the previous result and analysis using the test function
\begin{align*}
    G' = \big[ (\alpha')^2 G_0 + \mu' \big] \cdot G_2
\end{align*}
(where we determine $\mu'$ later in a similar manner) yields a similar result:

\begin{propn}
\label{propn-G_k+1-k=1}
    Let $k = 1$ and $p \geq 3$. Under the assumptions \eqref{eqn-ass-1} - \eqref{eqn-ass-3}, there exists a positive constant $M'$ such that if
    \begin{align*}
        \alpha' < M' \cdot \frac{B_{min}^2}{a_0 B_{max}^2 C_0 p},
    \end{align*}
    then
    \begin{align*}
        |D^2 \Rm|_{L^{2p}}, |D^3 T|_{L^{2p}} \text{ and } |D^2 F|_{L^{2p}}
    \end{align*}
    are uniformly bounded along the flow.
\end{propn}

\begin{proof}
    Using the pointwise estimate
    \begin{align*}
        \del_t G_2 &\leq \Big( \frac{1}{2 |\Omega|_\omega} \Big) \cdot \Big( \frac{1}{2} \Delta_R \big( G_2 \big) \Big) \\
        &\qquad + \sum_{m+l = 2} 2 \alpha' \Re \Big( \nablabar_{\overline{j}} \Binner{\Big( \frac{1}{2 |\Omega|_\omega} \Big) \nabla^{m+1} \nablabar^l (\Rm * \Rm)}{\nabla^m \nablabar^l \Rm}^{\overline{j}} \Big) \\
        &\qquad + \sum_{m+l = 2} 2 \alpha' \Re \Big( \nablabar_{\overline{j}} \Binner{\Big( \frac{1}{2 |\Omega|_\omega} \Big) \nabla^{m+1} \nablabar^l (F * F)}{\nabla^m \nablabar^l \Rm}^{\overline{j}} \Big) \\
        &\qquad + \sum_{m' + l' = 3} 2 \alpha' \Re \Big( \nabla_i \Binner{\Big( \frac{1}{2 |\Omega|_\omega} \nabla^{m'} \nablabar^{l'} (\Rm * \Rm)}{\nabla^{m'} \nablabar^{l'} T}^i \Big) \\
        &\qquad + \sum_{m' + l' = 3} 2 \alpha' \Re \Big( \nabla_i \Binner{\Big( \frac{1}{2 |\Omega|_\omega} \nabla^{m'} \nablabar^{l'} (F * F)}{\nabla^{m'} \nablabar^{l'} T}^i \Big) \\
        &\qquad + C \eps^{-1} \big( 1 + G_1^2 + G_2 \big) + \big[ C \eps + 42 a_0 B_{max} C_0^{-1} \alpha' \big] G_1 \cdot G_2 + \big[ C \eps + 20 a_0 B_{max} C_0 \alpha' - B_{min} \big] G_3,
    \end{align*}
    we obtain that
    \begin{align*}
        \del_t G' &\leq \Big( \frac{1}{2 |\Omega|_\omega} \Big) \cdot \Big( \frac{1}{2} \Delta_R \big( G' \big) \Big) - \Big( \frac{1}{2 |\Omega|_\omega} \Big) \cdot (\alpha')^2 \Re \inner{\nabla G_0}{\nabla G_2} \\
        &\qquad + C \eps^{-1} \big( 1 + G_1^2 + G' \big) + 2 \Re \mathbf{(A')} + 2 \Re \mathbf{(B')} \\
        &\qquad + \big[ C \eps + 134 a_0 B_{max} C_0 (\alpha')^3 + 42 a_0 B_{max} C_0^{-1} \alpha' \mu' - B_{min} (\alpha')^2 \big] G_1 \cdot G_2 \\
        &\qquad + \big[ C \eps + 60 a_0 B_{max} C_0^3 (\alpha')^3 + 20 a_0 B_{max} C_0 \alpha' \mu' - B_{min} \mu' \big] G_3,
    \end{align*}
    and
    \begin{align*}
        \del_t \bigg( \int_X (G')^p \, \vol \bigg) &\leq \int_X p (G')^{p-1} \cdot \Big( \frac{1}{2 |\Omega|_\omega} \Big) \cdot \Big( \frac{1}{2} \Delta_R \big( G' \big) \Big) \, \vol \\
        &\qquad + (\alpha')^2 \Re \bigg( \int_X p (G')^{p-1} \cdot \Big( \frac{1}{2 |\Omega|_\omega} \Big) \cdot \inner{\nabla G_0}{\nabla G_2} \, \vol \bigg) + C \eps^{-1} \int_X (G')^{p-1} \cdot \big( 1 + G_1^2 + G' \big) \, \vol \\
        &\qquad + 2 \Re \bigg( \int_X p (G')^{p-1} \cdot \mathbf{(A')} \, \vol \bigg)  + 2 \Re \bigg( \int_X p (G')^{p-1} \cdot \mathbf{(B')} \, \vol \bigg) \\
        &\qquad + \big[ C \eps + 134 a_0 B_{max} C_0 (\alpha')^3 + 42 a_0 B_{max} C_0^{-1} \alpha' \mu' - B_{min} (\alpha')^2 \big] \int_X p (G')^{p-1} \cdot G_1^2 \, \vol \\
        &\qquad + \big[ C \eps + 60 a_0 B_{max} C_0^3 (\alpha')^3 + 20 a_0 B_{max} C_0 \alpha' \mu' - B_{min} \mu' \big] \int_X p (G')^{p-1} \cdot G_3 \, \vol.
    \end{align*}
    where the $\mathbf{(A')}$ and $\mathbf{(B')}$ terms are defined analogously to before.

    We can check that
    \begin{align*}
        &(\alpha') \Re \bigg( \int_X p (G')^{p-1} \cdot \Big( \frac{1}{2 |\Omega|_\omega} \Big) \cdot \inner{\nabla G_0}{\nabla G_1} \, \vol \bigg) \\
        &\leq C \eps^{-1} \int_X (G')^{p-1} \cdot \big( 1 + G_1^2 + G' \big) \, \vol + C \eps \int_X (G')^{p-1} \cdot G_1 \cdot G_2 \, \vol + C \eps \int_X (G')^{p-1} \cdot G_3 \, \vol \\
        &\qquad + B_{max} (\alpha')^2 p \int_X (G')^{p-1} \cdot \big( 2 |\Rm| \cdot |D \Rm| + 4 |DT| \cdot |D^2 T| + 2 |F| \cdot |DF| \big) \\
        &\qquad \qquad \qquad \qquad \qquad \qquad \qquad \cdot \big( 6 |D^2 \Rm| \cdot |D^3 \Rm| + 8 |D^3 T| \cdot |D^4 T| + 6 |D^2 F| \cdot |D^3 F| \big) \, \vol \\
        &\leq  C \eps^{-1} \int_X (G')^{p-1} \cdot \big( 1 + G_1^2 + G' \big) \, \vol \\
        &\qquad + \big[ C \eps + \frac{1}{2} B_{min} (\alpha')^2 p \big] \int_X (G')^{p-1} \cdot G_1 \cdot G_2 \, \vol + \big[ C \eps + 768 B_{max}^2 B_{min}^{-1} C_0^2 (\alpha')^2 p \big] \int_X (G')^{p-1} \cdot G_3 \, \vol.
    \end{align*}

    Further the $\mathbf{(A')}$ term is bounded above by
    \begin{align*}
        &2 \Re \bigg( \int_X p (G')^{p-1} \cdot \mathbf{(A')} \, \vol \bigg) \\
        &\leq C \eps^{-1} \int_X (G')^p \, \vol \\
        &\qquad + \big[ C \eps + 6 a_0 B_{max} C_0^2 (\alpha')^4 p (p-1) + 48 a_0 B_{max} C_0^2 (\alpha')^4 p \big] \int_X (G')^{p-1} \cdot G_1 \cdot G_2 \, \vol \\
        &\qquad + 96 a_0 B_{max} C_0^2 (\alpha')^2 p \int_X (G')^{p-1} \cdot G_3 \, \vol + 12 a_0 B_{max} C_0^2 (\alpha')^2 (\mu')^{-1} p (p-1) \int_X (G')^{p-2} \cdot |\nabla G'|^2 \, \vol,
    \end{align*}
    while the $\mathbf{(B')}$ term is bounded above by
    \begin{align*}
        &2 \Re \bigg( \int_X p (G')^{p-1} \cdot \mathbf{(B')} \, \vol \bigg) \\
        &\leq C \eps^{-1} \int_X (G')^p \, \vol \\
        &\qquad + \big[ C \eps + 14112 a_0^2 B_{max}^2 B_{min}^{-1} (\alpha')^2 \mu' p (p-1) + 42336 a_0^2 B_{max}^2 B_{min}^{-1} C_0^2 (\alpha')^4 p (p-1) \\
        &\qquad \qquad + 504 a_0 B_{max} C_0 (\alpha')^3 p + 112 a_0 B_{max} C_0^2 (\alpha')^4 p \big] \int_X (G')^{p-1} \cdot G_1 \cdot G_2 \, \vol \\
        &\qquad + \big[ C \eps + 1568 a_0^2 B_{max}^2 B_{min}^{-1} C_0^2 (\alpha')^2 \mu' p (p-1) + 4704 a_0^2 B_{max}^2 B_{min}^{-1} C_0^4 (\alpha')^4 p (p-1) \\
        &\qquad \qquad + 112 a_0 B_{max} C_0^2 (\alpha')^2 p \big] \int_X (G')^{p-1} \cdot G_3 \, \vol \\
        &\qquad + \big[ \frac{1}{2} B_{min} p (p-1) \big] \int_X (G')^{p-2} \cdot |\nabla G'|^2 \, \vol.
    \end{align*}

    Which leaves us with the system of inequalities
    \begin{gather*}
        14112 a_0^2 B_{max}^2 B_{min}^{-1} \mu' (p-1) + 42336 a_0^2 B_{max}^2 B_{min}^{-1} C_0^2 (\alpha')^2 (p-1) + 42 a_0 B_{max} C_0^{-1} (\alpha')^{-1} \mu \\
        + 638 a_0 B_{max} C_0 \alpha' + 6 a_0 B_{max} C_0^2 (\alpha')^2 (p-1) + 160 a_0 B_{max} C_0^2 (\alpha')^2 < \frac{1}{2} B_{min}, \numberthis \label{eqn-system-4} \\[5pt]
        1568 a_0^2 B_{max}^2 B_{min}^{-1} C_0^2 (\alpha')^2 (p-1) + 4704 a_0^2 B_{max}^2 B_{min}^{-1} C_0^4 (\alpha')^4 \mu^{-1} (p-1) + 208 a_0 B_{max} C_0^2 (\alpha')^2 \mu^{-1} \\
        + 60 a_0 B_{max} C_0^3 (\alpha')^3 \mu^{-1} + 20 a_0 B_{max} C_0 \alpha' + 768 B_{max}^2 B_{min}^{-1} C_0^2 (\alpha')^2 \mu^{-1} < B_{min}, \numberthis \label{eqn-system-5} \\[5pt]
        12 a_0 B_{max} C_0^2 (\alpha')^2 \mu^{-1} < \frac{1}{2} B_{min}. \numberthis \label{eqn-system-6}
    \end{gather*}

    Assuming
    \begin{align*}
        \alpha' < M' \cdot \frac{B_{min}^2}{a_0 B_{max}^2 C_0 p} \text{ and } \mu = N' \cdot \frac{B_{min} C_0 \alpha'}{a_0 B_{max}} < M' N' \cdot \frac{B_{min}^3}{a_0^2 B_{max}^3 p}
    \end{align*}
    for some positive constants $M'$ and $N'$, the system \eqref{eqn-system-1} - \eqref{eqn-system-3} holds if
    \begin{gather*}
        638 M' + 42 N' + 14112 M' N' + 42502 (M')^2 < \frac{1}{2}, \numberthis \label{eqn-system-b-4} \\[5pt]
        4704 \frac{(M')^3}{N'} + 60 \frac{(M')^2}{N'} + 976 \frac{M'}{N'} + 20 M + 1568 (M')^2 < 1, \numberthis \label{eqn-system-b-5} \\[5pt]
        12 \frac{M'}{N'} < \frac{1}{2}. \numberthis \label{eqn-system-b-6}
    \end{gather*}

    One can solve this system by, for example, setting $M' = 1.196 \times 10^{-5}$ and $N' = 1.1676 \times 10^{-2}$. Gr\"{o}nwall's inequality then implies that $|G'|_{L^p}$ and hence $|D^2 \Rm|_{L^{2p}}$, $|D^3 T|_{L^{2p}}$ and $|D^2 F|_{L^{2p}}$ are uniformly bounded.
\end{proof}

\begin{rmk}
    The grouping and use of Young's inequality in the previous results are likely non-optimal. It may be of particular interest to find larger values of $M$ and $M'$ that work for computational purposes.
\end{rmk}

Altogether, under the assumptions \eqref{eqn-ass-1} - \eqref{eqn-ass-3} for any $k \geq 1$, we have that if $\alpha'$ is sufficiently small, then we have bounds for the pairs $|D^k \Rm|_{L^{2p}}$ and $|D^{k+1} \Rm|_{L^{2p}}$, $|D^{k+1} \Rm|_{L^{2p}}$ and $|D^{k+2} T|_{L^{2p}}$, and $|D^k F|_{L^{2p}}$ and $|D^{k+1} F|_{L^{2p}}$. Since the (real) dimension of $X$ is $6$, if $p \geq 3$, we may use the Sobolev embedding theorem to turn these $L^{2p}$-bounds into pointwise ones again (as per \cite{Ham82}), and hence $|D^k \Rm|, |D^{k+1} T|$, and $|D^k F|$ are bounded. We can thus bootstrap to show that all derivatives of $\Rm$, $T$, and $F$ are bounded pointwise.

Putting everything together and applying a standard result to extend the flow (see for example \cite{PPZ18b, Sua25}), we get the main result:

\thmmain*

\section{Further Questions}
\label{sect-qns}

Since the (modified) coupled flow \eqref{eqn-anomaly-flow} evolves the Hermitian metric $H$ by the Donaldson heat flow, one might ask whether we can obtain similar results on the regularity of $H$. For example, Donaldson \cite{Don85} showed that the flow has a unique smooth solution for all time without any prior assumptions on $H$ of its curvature $F$.

To achieve such a result, certain difficulties need to be overcome, such as the base metric being non-K\"{a}hler in general and also that the base metric is evolving in tandem with the bundle metric. The former can likely be dealt with using methods from \cite{Zha05, CLZZ26}. The latter, however, seems to be a larger issue as the flow is no longer the gradient flow for the Yang--Mills energy and thus a direct adaptation of Donaldson's results do not work. It is possible that the gradient flow viewpoint of the anomaly flow in \cite{AMP24} can remedy this.

\bibliographystyle{alpha} 
\bibliography{bib} 

@article{AMP24, 
title={Geometric flows and supersymmetry}, 
volume={405}, 
DOI={10.1007/s00220-023-04910-7}, 
number={1}, 
journal={Communications in Mathematical Physics}, 
author={Ashmore, Anthony and Minasian, Ruben and Proto, Yann}, 
year={2024}, 
month={Jan},
note={DOI: \href{https://doi.org/10.1007/s00220-023-04910-7}{10.1007/s00220-023-04910-7}},
}

@article{BV22, 
title={On the stability of the anomaly flow},
volume={29},
DOI={10.4310/mrl.2022.v29.n2.a1},
number={2},
journal={Mathematical Research Letters}, 
author={Bedulli, Lucio and Vezzoni, Luigi},
year={2022}, 
pages={323–338},
note={DOI: \href{https://doi.org/10.4310/mrl.2022.v29.n2.a1}{10.4310/mrl.2022.v29.n2.a1}},
}

@misc{CLZZ26,
      title={Long-time behavior of the {H}ermitian--{Y}ang--{M}ills flow on non-{K}\"ahler manifolds}, 
      author={Zeng Chen and Chao Li and Chuanjing Zhang and Xi Zhang},
      year={2026},
      eprint={2601.05614},
      archivePrefix={arXiv},
      primaryClass={math.DG},
      url={https://arxiv.org/abs/2601.05614}, 
      note={arXiv: \href{https://arxiv.org/abs/2601.05614}{2601.05614}},
}

@article{Don85,
author = {Donaldson, Simon K.},
title = {Anti Self-Dual {Y}ang--{M}ills Connections Over Complex Algebraic Surfaces and Stable Vector Bundles},
journal = {Proceedings of the London Mathematical Society},
volume = {s3-50},
number = {1},
pages = {1-26},
doi = {https://doi.org/10.1112/plms/s3-50.1.1},
url = {https://londmathsoc.onlinelibrary.wiley.com/doi/abs/10.1112/plms/s3-50.1.1},
eprint = {https://londmathsoc.onlinelibrary.wiley.com/doi/pdf/10.1112/plms/s3-50.1.1},
year = {1985},
note = {DOI: \href{https://doi.org/10.1112/plms/s3-50.1.1}{10.1112/plms/s3-50.1.1}},
}

@article{FHP21,
    author = {Fei, Teng and Huang, Zhijie and Picard, S\'{e}bastien},
    title = {The Anomaly Flow over {R}iemann Surfaces},
    journal = {International Mathematics Research Notices},
    volume = {2021},
    number = {3},
    pages = {2134-2165},
    year = {2021},
    month = {02},
    issn = {1073-7928},
    doi = {10.1093/imrn/rnz076},
    url = {https://doi.org/10.1093/imrn/rnz076},
    eprint = {https://academic.oup.com/imrn/article-pdf/2021/3/2134/36114108/rnz076.pdf},
    note = {DOI: \href{https://doi.org/10.1093/imrn/rnz076}{10.1093/imrn/rnz076}},
}

@article{FP21, 
title={Anomaly flow and {T}-duality}, 
volume={17}, 
DOI={10.4310/pamq.2021.v17.n3.a11},
number={3},
journal={Pure and Applied Mathematics Quarterly}, 
author={Fei, Teng and Picard, S\'{e}bastien}, 
year={2021}, 
pages={1083–1112},
note={DOI: \href{https://doi.org/10.4310/pamq.2021.v17.n3.a11}{10.4310/pamq.2021.v17.n3.a11}},
}

@misc{GFMPS25,
      title={A parabolic flow for the large volume heterotic {$G_2$} system}, 
      author={Mario Garcia-Fernandez and Andres J. Moreno and Alec Payne and Jeffrey Streets},
      year={2025},
      eprint={2512.14317},
      archivePrefix={arXiv},
      primaryClass={math.DG},
      url={https://arxiv.org/abs/2512.14317}, 
      note={arXiv: \href{https://arxiv.org/abs/2512.14317}{2512.14317}},
}

@article{Ham82, 
    title={Three-manifolds with positive {R}icci curvature}, 
    volume={17}, DOI={10.4310/jdg/1214436922}, 
    number={2}, 
    journal={Journal of Differential Geometry}, 
    author={Hamilton, Richard S.}, 
    year={1982}, 
    month={Jan},
    note={DOI: \href{https://doi.org/10.1007/s10455-020-09710-z}{10.1007/s10455-020-09710-z}},
}

@article{Hul86,
title = {Compactifications of the heterotic superstring},
journal = {Physics Letters B},
volume = {178},
number = {4},
pages = {357-364},
year = {1986},
issn = {0370-2693},
doi = {https://doi.org/10.1016/0370-2693(86)91393-6},
url = {https://www.sciencedirect.com/science/article/pii/0370269386913936},
author = {Christopher M. Hull},
note  = {DOI: \href{https://doi.org/10.1016/0370-2693(86)91393-6}{10.1016/0370-2693(86)91393-6}},
}

@article{KPS26, 
    title={Flows of conformally coclosed {$G_2$}-structures with dilaton}, 
    volume={36}, 
    DOI={10.1007/s12220-026-02416-x}, 
    journal={The Journal of Geometric Analysis}, 
    author={Karigiannis, Spiro and Picard, Sébastien and Suan, Caleb}, 
    year={2026}, 
    month={Apr},
    note={DOI: \href{https://doi.org/10.1007/s12220-026-02416-x}{10.1007/s12220-026-02416-x}},
}

@article{Puj21,
title = {The {H}ull--{S}trominger system and the Anomaly flow on a class of solvmanifolds},
journal = {Journal of Geometry and Physics},
volume = {170},
pages = {104352},
year = {2021},
issn = {0393-0440},
doi = {https://doi.org/10.1016/j.geomphys.2021.104352},
url = {https://www.sciencedirect.com/science/article/pii/S0393044021001984},
author = {Mattia Pujia},
note={DOI: \href{https://doi.org/10.1016/j.geomphys.2021.104352}{10.1016/j.geomphys.2021.104352}}
}

@article{PPZ18a, 
title={The anomaly flow and the fu-yau equation},
volume={4}, DOI={10.1007/s40818-018-0049-9},
number={2}, 
journal={Annals of PDE},
author={Phong, Duong H. and Picard, S\'{e}bastien and Zhang, Xiangwen}, 
year={2018}, 
month={Jul},
note={DOI: \href{https://doi.org/10.1007/s40818-018-0049-9}{10.1007/s40818-018-0049-9}},
}

@article {PPZ18b,
    AUTHOR = {Phong, Duong H. and Picard, S\'{e}bastien and Zhang, Xiangwen},
     TITLE = {Anomaly flows},
   JOURNAL = {Comm. Anal. Geom.},
  FJOURNAL = {Communications in Analysis and Geometry},
    VOLUME = {26},
      YEAR = {2018},
    NUMBER = {4},
     PAGES = {955--1008},
      ISSN = {1019-8385,1944-9992},
   MRCLASS = {53C44 (32Q99 53C55 81T30)},
  MRNUMBER = {3853932},
MRREVIEWER = {Masaya\ Kawamura},
       DOI = {10.4310/CAG.2018.v26.n4.a9},
       URL = {https://doi.org/10.4310/CAG.2018.v26.n4.a9},
      NOTE = {DOI: \href{https://doi.org/10.4310/CAG.2018.v26.n4.a9}{10.4310/CAG.2018.v26.n4.a9}}
}

@article{PPZ18c, 
    title={Geometric flows and {S}trominger Systems}, 
    volume={288}, DOI={10.1007/s00209-017-1879-y}, 
    number={1–2}, 
    journal={Mathematische Zeitschrift}, 
    author={Phong, Duong H. and Picard, S\'{e}bastien and Zhang, Xiangwen}, 
    year={2018}, 
    month={Mar}, 
    pages={101–113},
    note={DOI: \href{https://doi.org/10.1007/s00209-017-1879-y}{10.1007/s00209-017-1879-y}}
}

@article{PPZ19a,
  title={The Anomaly flow on unimodular Lie groups},
  author={Phong, Duong H. and Picard, S\'{e}bastien and Zhang, Xiangwen},
  journal={Advances in complex geometry},
  volume={735},
  pages={217--237},
  year={2019},
  publisher={American Mathematical Society},
  eprint={1705.09763},
  archivePrefix={arXiv},
  primaryClass={math.DG}
}

@article{PU21, 
    title={The anomaly flow on Nilmanifolds}, 
    volume={60}, 
    DOI={10.1007/s10455-021-09781-6}, 
    number={3}, 
    journal={Annals of Global Analysis and Geometry}, 
    author={Pujia, Mattia and Ugarte, Luis}, 
    year={2021}, 
    month={Jun}, 
    pages={501–537},
    note= {DOI: \href{https://doi.org/10.1007/s10455-021-09781-6}{10.1007/s10455-021-09781-6}}
}

@article{Str86,
title = {Superstrings with torsion},
journal = {Nuclear Physics B},
volume = {274},
number = {2},
pages = {253-284},
year = {1986},
issn = {0550-3213},
doi = {https://doi.org/10.1016/0550-3213(86)90286-5},
url = {https://www.sciencedirect.com/science/article/pii/0550321386902865},
author = {Andrew Strominger},
note = {DOI: \href{https://doi.org/10.1016/0550-3213(86)90286-5}{10.1016/0550-3213(86)90286-5}},
}

@article{Sua25, 
    title={Anomaly flow: Shi-type estimates and long-time existence}, 
    DOI={10.1515/crelle-2025-0051}, 
    journal={Journal für die reine und angewandte Mathematik (Crelles Journal)}, 
    author={Suan, Caleb}, 
    year={2025}, 
    month={Jul},
    note={DOI: \href{https://doi.org/10.1515/crelle-2025-0051}{10.1515/crelle-2025-0051}}
}

@article{Zha05,
title = {Hermitian–-{E}instein metrics on holomorphic vector bundles over {H}ermitian manifolds},
journal = {Journal of Geometry and Physics},
volume = {53},
number = {3},
pages = {315-335},
year = {2005},
issn = {0393-0440},
doi = {https://doi.org/10.1016/j.geomphys.2004.07.002},
url = {https://www.sciencedirect.com/science/article/pii/S0393044004001111},
author = {Xi Zhang},
note = {DOI: \href{https://doi.org/10.1016/j.geomphys.2004.07.002}{10.1016/j.geomphys.2004.07.002}},
}

\end{document}